\documentclass[a4paper, 11pt]{amsproc}

\usepackage{fullpage}
\usepackage{amsmath, amsthm, amssymb, mathtools, mathrsfs, stmaryrd}
\usepackage{ascmac}
\usepackage{comment}
\usepackage{bm}

\usepackage{hyperref}

\usepackage{graphicx}
\usepackage{here}
\usepackage{time}
\usepackage[abbrev]{amsrefs}

\usepackage{xcolor}
\usepackage[capitalize,nameinlink,noabbrev,nosort]{cleveref}
\hypersetup{
	colorlinks=true,       
	linkcolor=blue,          
	citecolor=blue,        
	filecolor=blue,      
	urlcolor=blue,           
}

\makeatletter
\@namedef{subjclassname@2020}{%
  \textup{2020} Mathematics Subject Classification}
\makeatother

\newtheorem{theoremcounter}{Theorem Counter}[section]

\theoremstyle{definition}
\newtheorem{definition}[theoremcounter]{Definition}

\theoremstyle{plain}
\newtheorem{lemma}[theoremcounter]{Lemma}

\newtheorem{corollary}[theoremcounter]{Corollary}
\newtheorem{conjecture}[theoremcounter]{Conjecture}
\newtheorem{theorem}[theoremcounter]{Theorem}

\numberwithin{equation}{section}

\newcommand{\Z}{\mathbb{Z}}
\newcommand{\Q}{\mathbb{Q}}

\newcommand{\C}{\mathbb{C}}

\def\st#1#2{\genfrac{[}{]}{0pt}{}{#1}{#2}}

\begin{document}

\title[]{Curious congruences for cyclotomic polynomials II} 

\author{Toshiki Matsusaka}
\address{Faculty of Mathematics, Kyushu University,
Motooka 744, Nishi-ku, Fukuoka 819-0395, Japan}
\email{matsusaka@math.kyushu-u.ac.jp} 

\author{Genki Shibukawa}
\address{Department of Mathematics, Graduate School of Science, Kobe University,
1-1, Rokkodai, Nada-ku Kobe 657-8501, Japan}
\email{g-shibukawa@math.kobe-u.ac.jp} 




\maketitle

\begin{abstract}
	We promote the recent research by Akiyama and Kaneko on the higher-order derivative values $\Phi_n^{(k)}(1)$ of the cyclotomic polynomials. This article focuses on Lehmer's explicit formula of $\Phi_n^{(k)}(1)/\Phi_n(1)$ as a polynomial of the Euler and Jordan totient functions over $\Q$. Then we prove Akiyama--Kaneko's conjecture that the polynomials have a specific simple factor.
\end{abstract}


\section{Introduction}

The coefficients and higher-order derivatives of cyclotomic polynomials have long been engaging mathematicians' interests. The $n$-th cyclotomic polynomial
\[
	\Phi_n(x) := \prod_{\substack{0 < k \leq n \\ (k,n) = 1}} \left(x - e^{2\pi i k/n}\right)
\]
is the minimal polynomial of the $n$-th primitive roots of unity over $\Q$. Since collecting all $n$-th roots of unity yields $x^n - 1 = \prod_{d \mid n} \Phi_d(x)$, the M\"{o}bius inversion leads to the equivalent expression
\[
	\Phi_n(x) = \prod_{d \mid n} (x^d-1)^{\mu(n/d)},
\]
where $\mu$ is the M\"{o}bius function. The study on the higher-order derivative values of $\Phi_n(x)$ at $x=1$, which is the subject of this article, goes back to V.-A. Lebesgue in 1859 \cite{Lebesgue1859} and H\"{o}lder in 1936~\cite{Holder1936}. The general foundations of this subject were established by Lehmer in 1966 \cite{Lehmer1966}. Since then, it has been studied continuously and extensively, as summarized in \cite{HerreraPoyatosMoree2021} and \cite{Sanna2022}.

Recently, Akiyama and Kaneko~\cite[Theorem 3]{AkiyamaKaneko2022} discovered ``curious congruences" for odd-order derivatives of cyclotomic polynomials. Let $\phi(n) = \deg \Phi_n(x)$ be the Euler totient function. Then, $2\Phi_n^{(3)}(1)$ is divisible by $\phi(n)-2$. Moreover, for $k > 1$, $\Phi_n^{(2k+1)}(1)$ is divisible by $\phi(n) - 2k$. 

According to Lehmer~\cite{Lehmer1966}, there explicitly exist polynomials $F_k(x_1, \dots, x_k) \in \Q[x_1, \dots, x_k]$ such that $\Phi_n^{(k)}(1)/\Phi_n(1) = F_k(\phi(n)/2, J_2(n)/4, \dots, J_k(n)/2k)$, where $J_k(n)$ is Jordan's totient function. The polynomials will be discussed in detail in~\cref{Section-3}, then Akiyama and Kaneko posed the following conjecture as an analogue of the above congruences.

\begin{conjecture}[{\cite[Conjecture 1]{AkiyamaKaneko2022}}]\label{main1}
	For any non-negative integer $k$, $F_{2k+1}(x_1, \dots, x_{2k+1})$ is divisible by $x_1 - k$ in $\Q[x_1, \dots, x_{2k+1}]$.
\end{conjecture}

In this article, we prove this conjecture. The divisibility suggests a direct relationship with Akiyama--Kaneko's congruences, but it should be noted that \cref{main1} does not immediately yield them. In \cref{Section-4}, we penetrate a certain integrality property of the polynomial $F_k$ and give a new proof for the curious congruences as a corollary.

\section{Preliminaries}

We review the basic functions used in this article. Jordan's totient functions $J_k(n)$ are a generalization of the Euler totient function. We give its two equivalent definitions,
\begin{align}\label{Jordan-def}
	J_k(n) = n^k \prod_{p \mid n} (1-p^{-k}) = \sum_{d \mid n} \mu(n/d) d^k,
\end{align}
where $p$ runs over prime divisors of $n$ and $d$ does over all divisors of $n$. We easily see that $J_k(n)$ is even for any $k \geq 1$ and $n \geq 3$. The special case $\phi(n) = J_1(n)$ is the Euler totient function. For the details, see \cite[Chapter V.3]{Sivaramakrishnan1989}.

The Bernoulli numbers $B_n$ are defined by the generating series
\begin{align}\label{Bernoulli-def}
	\frac{t}{e^t - 1} = \sum_{n=0}^\infty B_n \frac{t^n}{n!},
\end{align}
or
\begin{align}\label{Bernoulli-int}
	\log \left(\frac{\sinh(t/2)}{t/2} \right) = \sum_{n=2}^\infty \frac{B_n}{n} \frac{t^n}{n!}
\end{align}
with $B_0 = 1$ and $B_1 = -1/2$.

The coefficients of the factorial notation $(x)_n = x(x-1)(x-2) \cdots (x-n+1)$ define the (unsigned) Stirling numbers of the first kind $\st{\cdot}{\cdot}$ by
\begin{align}\label{Stirling-def}
	(x)_n = \sum_{k=0}^n (-1)^{n+k} \st{n}{k} x^k.
\end{align}
It is also known that the exponential generating series of the Stirling numbers is given by
\begin{align}\label{Stirling-gen}
	\frac{(-\log(1-t))^k}{k!} = \sum_{n=k}^\infty \st{n}{k} \frac{t^n}{n!}
\end{align}
for $k \geq 0$. 
The generating function for the factorial notations is the following (binomial theorem): 
\begin{align}\label{Pochhammer}
	(1+t)^x = \sum_{n=0}^\infty (x)_n \frac{t^n}{n!} = \sum_{n=0}^\infty \binom{x}{n} t^n,
\end{align}
where $\binom{x}{n}:=\frac{(x)_{n}}{n!}$ is the binomial coefficient.

The basic properties of the Stirling numbers are expounded in great detail in \cite{ArakawaIbukiyamaKaneko2014}.

\section{Lehmer's observation and a proof of \cref{main1}}\label{Section-3}

For any $k \geq 1$, we introduce the polynomials $s_k(x_1, \dots, x_k) \in \Q[x_1, \dots, x_k]$ and $F_k(x_1, \dots, x_k) \in \Q[x_1, \dots, x_k]$ by
\begin{align}\label{sk-def}
	s_k(x_1, \dots, x_k) &:= \frac{2 (-1)^{k-1}}{(k-1)!} \sum_{m=1}^k B_m \st{k}{m} x_m , 
\end{align}
and 
\begin{align}\label{Fk-def}
	F_k(x_1, \dots, x_k) := k! \sum_{\substack{\lambda_1, \dots, \lambda_k \geq 0 \\ \lambda_1 + 2\lambda_2 + \cdots + k \lambda_k = k}} \prod_{j=1}^k \frac{(-s_j(x_1, \dots, x_j))^{\lambda_j}}{\lambda_j! j^{\lambda_j}}, \quad F_0 := 1, 
\end{align}
respectively. The first few examples $s_k(x_1, \dots, x_k)$ and $F_k(x_1, \dots, x_k)$ are given by
\begin{align*}
	s_1(x_1) &= - x_1,\\
	s_2(x_1, x_2) &= \frac{1}{3}(3x_1 - x_2),\\
	s_3(x_1, x_2, x_3) &= - \frac{1}{2}(2x_1 - x_2),\\
	s_4(x_1, \dots, x_4) &= \frac{1}{90}(90x_1 - 55x_2 + x_4),\\
	s_5(x_1, \dots, x_5) &= - \frac{1}{36}(36x_1 - 25x_2 + x_4),
\end{align*}
and 
\begin{align*}
	F_1(x_1) &= x_1,\\
	F_2(x_1, x_2) &= \frac{1}{3}(3x_1^2 - 3x_1 + x_2),\\
	F_3(x_1, x_2, x_3) &= (x_1-1)(x_1^2 - 2x_1 + x_2),\\
	F_4(x_1, \dots, x_4) &= \frac{1}{15} (15x_1^4 - 90x_1^3 + (30x_2 + 165) x_1^2 - (90x_2+90)x_1 + 5x_2^2 + 55x_2 - x_4),\\
	F_5(x_1, \dots, x_5) &= \frac{1}{3}(x_1 - 2) (3x_1^4 - 24x_1^3 + (10x_2 + 57)x_1^2 - (40x_2+36)x_1 +5x_2^2 + 25x_2 - x_4).
\end{align*}
We note that odd-numbered variables other than $x_1$ do not appear since the Bernoulli numbers satisfy $B_n = 0$ for any odd $n \geq 3$. 

By the generating function \eqref{Stirling-gen}, and the definitions of the polynomials $s_k(x_1, \dots, x_k)$ and $F_k(x_1, \dots, x_k)$, we have the following generating functions.

\begin{lemma}\label{generating-thm}
	We define the functions $Q(\bm{x}, t)$ and $P(\bm{x}, t)$ by
	\begin{align}\label{s-gen exp}
	\begin{split}
		Q(\bm{x}, t) &:= - 2 \sum_{n=1}^\infty \frac{B_{n}}{n!}  (-\log(1+t))^{n} x_{n}\\
			&= -\log(1+t) x_{1} - 2 \sum_{\nu=1}^\infty \frac{B_{2\nu}}{(2\nu)!}  (-\log(1+t))^{2\nu} x_{2\nu} \in \Q[x_1, x_2, x_4, x_6, \dots] \llbracket t \rrbracket
	\end{split}	
	\end{align}
	and 
	\begin{align}\label{F-gen exp}
	\begin{split}
		P(\bm{x}, t) &:= \exp \left(- Q(\bm{x}, t) \right)\\
			&= (1+t)^{x_1} \exp \left(2 \sum_{\nu=1}^\infty \frac{B_{2\nu}}{(2\nu)!}  (-\log(1+t))^{2\nu} x_{2\nu} \right) \in \Q[x_1, x_2, x_4, x_6, \dots] \llbracket t \rrbracket .
	\end{split}
	\end{align}
	Then, we have 
	\begin{align}\label{s-gen}
		Q(\bm{x}, t) &= \sum_{j=1}^\infty \frac{s_j(x_1, \dots, x_j) t^j}{j}, \\
		\label{F-gen}
		P(\bm{x}, t) &= \sum_{k=0}^\infty F_k(x_1, \dots, x_k) \frac{t^k}{k!}.
	\end{align}
\end{lemma}

In 1966, Lehmer~\cite[Theorem 3]{Lehmer1966} showed, for $n \geq 2$, that the derivative values $\Phi_n^{(k)}(1)/\Phi_n(1) \in \Z$ are expressed as a polynomial on the Jordan totient functions $J_1(n), \dots, J_k(n)$ over $\Q$ (see also \cite{AkiyamaKaneko2022}). 

\begin{lemma}\label{Lehmer-thm}
	For any $n \geq 2$, we have 
	\[
		\frac{\Phi_n^{(k)}(1)}{\Phi_n(1)} = F_k \left(\frac{\phi(n)}{2}, \frac{J_2(n)}{4}, \dots, \frac{J_k(n)}{2k} \right).
	\]
\end{lemma}

Besides the original proof by Lehmer, another proof by Herrera-Poyatos--Moree~\cite{HerreraPoyatosMoree2021} has been known recently. 

To prove \cref{main1}, we first provide a proof for the unproven Lehmer observation given in \cite[Section 7]{Lehmer1966}. First, the polynomials $\Omega_m$, for which Lehmer observed only the first few terms, are generally determined as follows.

\begin{definition}\label{Omega-def}
	We define the polynomials $\Omega_m(x_2, x_4, \dots, x_{2m}) \in \Q[x_2, x_4, \dots, x_{2m}]$ by the generating series
	\begin{align*}
		1 + 2 \sum_{m=1}^\infty \frac{B_{2m}}{(2m)!} \Omega_m(x_2, \dots, x_{2m}) u^{2m} := \exp \left(2 \sum_{\nu=1}^\infty \frac{B_{2\nu}}{(2\nu)!} \left(2 \sinh^{-1} \left(\frac{u}{2}\right)\right)^{2\nu} x_{2\nu} \right).
	\end{align*}
\end{definition}

The first few examples of the polynomials $\Omega_m$ are
\begin{align*}
	\Omega_1(x_2) &= x_2,\\
	\Omega_2(x_2, x_4) &= x_4 - 5x_2(x_2-1),\\
	\Omega_3(x_2, x_4, x_6) &= x_6 - 7x_4(x_2-1) + \frac{35}{3} x_2(x_2 - 1)(x_2 - 2) + \frac{14}{3} x_2,
\end{align*}
which coincide with the table in \cite{Lehmer1966}.

\begin{theorem}\label{Lehmer-obs}
	We have
	\[
		F_k(x_1, \dots, x_k) = (x_1)_k + 2 \sum_{m=1}^{\lfloor k/2 \rfloor} B_{2m} {k \choose 2m} (x_1 - m)_{k-2m} \Omega_m(x_2, \dots, x_{2m}).
	\]
\end{theorem}

\begin{proof}
    Since
	\[
		\log(1+t) = 2 \sinh^{-1} \left(\frac{t}{2 \sqrt{1+t}} \right)
	\]
	holds, comparing with the definition of $\Omega_m$ and the generating function for $F_k$ \eqref{F-gen exp}, \eqref{F-gen} implies
	\[
		P(\bm{x}, t) = (1+t)^{x_1} + 2 \sum_{m=1}^\infty \frac{B_{2m}}{(2m)!} \Omega_m(x_2, \dots, x_{2m}) t^{2m} (1+t)^{x_1-m}.
	\]
	Finally, \eqref{Pochhammer} gives
	\[
		P(\bm{x}, t) = \sum_{k=0}^\infty \left((x_1)_k + 2 \sum_{m=1}^{\lfloor k/2 \rfloor} B_{2m} {k \choose 2m} (x_1-m)_{k-2m} \Omega_m(x_2, \dots, x_{2m}) \right) \frac{t^k}{k!},
	\]
	which concludes the proof.
\end{proof}

\begin{corollary}[{\cref{main1}}]\label{main-cor}
	For any non-negative integer $k$, $F_{2k+1}(x_1, \dots, x_{2k+1})$ is divisible by $x_1 - k$ in $\Q[x_1, \dots, x_{2k+1}]$.
\end{corollary}

\begin{proof}
	In the expression of $F_{2k+1}(x_1, \dots, x_{2k+1})$ in \cref{Lehmer-obs}, the factorial notations $(x_1)_{2k+1}$ and $(x_1 - m)_{2k+1-2m}$ are divisible by $x_1 - k$.
\end{proof}

\section{Akiyama--Kaneko's congruences}\label{Section-4}

Akiyama--Kaneko discovered congruences on the odd-order derivative values $\Phi_n^{(2k+1)}(1) \in \Z$.

\begin{theorem}[{\cite[Theorem 3]{AkiyamaKaneko2022}}]\label{AK-theorem}
	We have the following.
	\begin{enumerate}
		\item $2\Phi_n^{(3)}(1)$ is divisible by $\phi(n) - 2$.
		\item Suppose that $k \geq 2$. Then $\Phi_n^{(2k+1)}(1)$ is divisible by $\phi(n) - 2k$.
	\end{enumerate}
\end{theorem}

More strongly, they observed that $\Phi_n^{(2k+1)}(1)/\Phi_n(1)$ is likely to be divisible by $\phi(n) - 2k$ for $k \geq 2$. Their congruences suggest a direct relationship with the divisibility shown in the previous section, but it should be noted that \cref{main-cor} does not immediately yield \cref{AK-theorem}. Our second theorem (\cref{main2}) gives a certain integrality property of the polynomials $F_k(x_1, \dots, x_k)$, which relates them directly. We begin by preparing useful series of lemmas.

\begin{lemma}
	For any integer $n \geq 1$, we define the polynomial $V_n(x) \in \Q[x]$ by
	\begin{align}\label{V-def}
		V_n(x) := \prod_{1 \leq k < n/2} \left(x^2 + 4 \sin^2 \left(\frac{\pi k}{n}\right) \right)
	\end{align}
	with $V_1(x) = V_2(x) = 1$. Then $V_n(x) \in \Z[x]$ and we have
	\begin{align}\label{V-Chebyshev}
		\sinh(n \theta) = \begin{cases}
			\sinh \theta V_n(2\sinh \theta) &\text{if } n: \text{odd},\\
			\sinh 2\theta V_n(2\sinh \theta) &\text{if } n: \text{even}.
		\end{cases}
	\end{align}
\end{lemma}

\begin{proof}
	We first recall two kinds of Chebyshev polynomials $T_n(x), U_n(x) \in \Z[x]$ characterized by $T_n(\cos \theta) = \cos n\theta$ and $U_{n-1}(\cos \theta) = \sin n\theta /\sin \theta$. The characterizations immediately lead to the following.
	\begin{align*}
		\sinh(n \theta) &= \begin{cases}
			(-1)^{\frac{n+1}{2}} i T_n(i \sinh \theta) &\text{if } n: \text{odd},\\
			(-1)^{\frac{n}{2}} i \cosh \theta U_{n-1}(i \sinh \theta) &\text{if } n: \text{even}.
		\end{cases}
	\end{align*}
	Let define the polynomials $\widetilde{V}_n(x)$ by
	\[
		\widetilde{V}_n(x) := \begin{cases}
			\displaystyle{\frac{2(-1)^{\frac{n+1}{2}} i}{x} T_n \left(\frac{ix}{2}\right)} &\text{if } n: \text{odd},\\
			\displaystyle{\frac{(-1)^{\frac{n}{2}} i}{x} U_{n-1} \left(\frac{ix}{2}\right)} &\text{if } n: \text{even}.
		\end{cases}
	\]
	It is then enough to show that $\widetilde{V}_n(x) \in \Z[x]$ and $\widetilde{V}_n(x) = V_n(x)$. By the known explicit formulas for the Chebyshev polynomials 
	(see~\cite[Section 10.11 (22), (23)]{Erdelyi1953}), 
	we have
	\begin{align}\label{V-explicit}
		\widetilde{V}_n(x) = \begin{cases}
			\displaystyle{\sum_{m=0}^{\frac{n-1}{2}} \frac{n(n-m-1)!}{m!(n-2m)!} x^{n-2m-1}} &\text{if } n: \text{odd},\\
			\displaystyle{\sum_{m=0}^{\frac{n}{2}-1} {n-m-1 \choose m} x^{n-2m-2}} &\text{if } n: \text{even}.
		\end{cases}
	\end{align}
	Since $\frac{n(n-m-1)!}{m!(n-2m)!} = {n-m \choose m} + {n-m-1 \choose m-1}$, we see that $\widetilde{V}_n(x) \in \Z[x]$. As for the coincidence, since both $V_n(x)$ and $\widetilde{V}_n(x)$ are monic polynomials, it suffices to check the locations of their zeros. For an odd $n$, the set of zeros of $\widetilde{V}_n(x)$ are given by $\{2i \cos \frac{(2k+1) \pi}{2n} \in \C \mid 0 \leq k < n\} \setminus \{0\}$, which coincides with the set $\{2i \sin \frac{\pi k}{n} \in \C \mid 0 < |k| \leq \frac{n-1}{2}\}$ of zeros of $V_n(x)$. A similar argument works for even $n$.
\end{proof}

\begin{lemma}\label{omega-rational-poly}
	For any integer $n \geq 1$, we define the polynomial $W_n(x) \in \Q[x]$ by
	\begin{align*}
		W_n(x) &:= \prod_{\substack{1 \leq k < n/2 \\ (k,n)=1}} \left(x^2 + 4 \sin^2 \left(\frac{\pi k}{n} \right) \right)
	\end{align*}
	with $W_1(x) = W_2(x) = 1$. Then for any $n \geq 3$, we have
	\[
		\frac{1}{\Phi_n(1)} W_n(x) = 1 + 2 \sum_{m=1}^\infty \frac{B_{2m}}{(2m)!} \omega_m(n) x^{2m},
	\]
	where we put $\omega_m(n) = \Omega_m(J_2(n)/4, \dots, J_{2m}(n)/2m) \in \Q$.
\end{lemma}

\begin{proof}
	By \cref{Omega-def} and \eqref{Jordan-def}, we have
	\begin{align*}
		1 + 2 \sum_{m=1}^\infty \frac{B_{2m}}{(2m)!} \omega_m(n) x^{2m} &= \exp \left( \sum_{m=2}^\infty \frac{B_m}{m!} \left(2 \sinh^{-1} \left(\frac{x}{2}\right)\right)^m \frac{J_m(n)}{m} \right)\\
			&= \prod_{d \mid n} \exp \left(\sum_{m=2}^\infty \frac{B_m}{m \cdot m!} \left(2d \sinh^{-1} \left(\frac{x}{2}\right) \right)^m \right)^{\mu(n/d)}.
	\end{align*}
	Moreover, by applying \eqref{Bernoulli-int}, \eqref{V-Chebyshev}, and the facts that
	\begin{align*}
		\sum_{\substack{d \mid n \\ d \equiv j\ (2)}} \mu \left(\frac{n}{d}\right) = 0, \quad \prod_{d \mid n} d^{\mu(n/d)} = \Phi_n(1)
	\end{align*}
	for $n \geq 3$ and $j \in \{0,1\}$ (see \cite[Remark 2]{AkiyamaKaneko2022}), we reach the simple expression
	\begin{align*}
		1 + 2 \sum_{m=1}^\infty \frac{B_{2m}}{(2m)!} \omega_m(n) x^{2m} &= \frac{1}{\Phi_n(1)} \prod_{d \mid n} \left(\sinh \left(d \sinh^{-1} \left(\frac{x}{2} \right) \right)\right)^{\mu(n/d)}\\
			&= \frac{1}{\Phi_n(1)} \prod_{d \mid n} V_d (x)^{\mu(n/d)}.
	\end{align*}
	By applying the M\"{o}bius inversion formula to $V_n(x) = \prod_{d \mid n} W_d(x)$, we get
	\begin{align}\label{Mobius-inver}
		\prod_{d \mid n} V_d(x)^{\mu(n/d)} = W_n(x),
	\end{align}
	which concludes the proof.
\end{proof}

\begin{lemma}\label{integrality}
	For any integer $n \geq 3$ and $1 \leq m < \phi(n)/2$, we have
	\[
		\frac{2B_{2m}}{(2m)!} \omega_m(n) \in \Z.
	\]
\end{lemma}

\begin{proof}
	By $V_n(x) = \prod_{d \mid n} W_d(x)$ and $V_n(x) \in \Z[x]$, Gauss' lemma implies $W_n(x) \in \Z[x]$. Thus it suffices to show that all coefficients of $W_n(x)$ except for the highest degree term $x^{\phi(n)}$ are divisible by $\Phi_n(1)$. We recall that the value $\Phi_n(1)$ is explicitly given by
	\[
		\Phi_n(1) = \begin{cases}
			p &\text{if } n = p^r,\ p\text{ is prime},\\
			1 &\text{if otherwise}.
		\end{cases}
	\]
	Thus we now assume that $n = p^r$ is a power of a prime number $p$. By \eqref{Mobius-inver}, we have $W_{p^r}(x)V_{p^{r-1}}(x) = V_{p^r}(x)$. By the explicit formula \eqref{V-explicit}, $V_{p^r}(x) \equiv x^{p^r-1} \pmod{p}$ for an odd $p$ and $V_{2^r}(x) \equiv x^{2^r-2} \pmod{2}$. Thus, for an odd prime $p$, we have $W_{p^r}(x) x^{p^{r-1}-1} \equiv x^{p^r-1} \pmod{p}$, which implies that $W_{p^r}(x) \equiv x^{\phi(p^r)} \pmod{p}$. The same argument works for $p=2$.
\end{proof}

\begin{theorem}\label{main2}
	For any integers $n \geq 3$ and $1 \leq k < \phi(n)$, the polynomial $F_{k,n}(x)$ defined by $F_{k,n}(x) = F_k(x, J_2(n)/4, \dots, J_k(n)/2k)$ is in $\Z[x]$.
\end{theorem}

\begin{proof}
	By \cref{Lehmer-obs}, the polynomial $F_{k,n}(x)$ is given by
	\begin{align}\label{Fkn-explicit}
		F_{k,n}(x) = (x)_k + \sum_{m=1}^{\lfloor k/2 \rfloor} \frac{k!}{(k-2m)!} \frac{2B_{2m}}{(2m)!} \omega_m(n) (x-m)_{k-2m}.
	\end{align}
	\cref{integrality} implies that, for any $1 \leq k < \phi(n)$, $F_{k,n}(x) \in \Z[x]$.
\end{proof}

Finally, we show Akiyama--Kaneko's observation, a refined version of \cref{AK-theorem} as a corollary of \cref{main2}. 

\begin{corollary}
	We have the following.
	\begin{enumerate}
		\item $2 \Phi_n^{(3)}(1)/\Phi_n(1)$ is divisible by $\phi(n)-2$.
		\item Suppose that $k \geq 2$. Then $\Phi_n^{(2k+1)}(1)/\Phi_n(1)$ is divisible by $\phi(n) - 2k$.
	\end{enumerate}
\end{corollary}

\begin{proof}
	Since $\deg \Phi_n(x) = \phi(n)$ is even for $n \geq 3$, it is sufficient to consider only the cases for $n \geq 3$ and $3 \leq 2k+1 < \phi(n)$. By \cref{main-cor} and \cref{main2}, $F_{2k+1, n}(x) \in \Z[x]$ is divisible by $x-k$. Therefore, $F_{2k+1,n}(\phi(n)/2) = \Phi_n^{(2k+1)}(1)/\Phi_n(1)$ is divisible by $\phi(n)/2 - k$. In particular, if $k \geq 2$, the expression \eqref{Fkn-explicit} yields that $F_{2k+1, n}(\phi(n)/2)$ is divisible by $2(\phi(n)/2-k)$.
\end{proof}

\section*{Acknowledgements}

It is a pleasure to thank Shigeki Akiyama and Hajime Kaneko for introducing the authors to this topic and for valuable discussions.
This work was supported by JSPS KAKENHI Grant Numbers JP20K14292, JP21K18141 (Matsusaka) and JP21K13808 (Shibukawa).

\bibliographystyle{amsplain}
\bibliography{References}

\end{document}